\documentclass[12pt]{article}

\usepackage{amsmath,amsthm,amssymb}
\usepackage{mathtools}
\usepackage{hyperref}
\usepackage{indentfirst}
\usepackage{geometry}

\theoremstyle{plain}
\newtheorem{theorem}{Theorem}
\newtheorem{lemma}[theorem]{Lemma}

\newtheorem*{main}{Main Theorem}

\DeclarePairedDelimiter\abs{\lvert}{\rvert}

\renewcommand{\ge}{\geqslant}
\renewcommand{\le}{\leqslant}

\title{Tomaszewski's problem on\\randomly signed sums, revisited}
\author{
Ravi B. Boppana\thanks{Department of Mathematics, Massachusetts Institute of Technology, Cambridge, Massachusetts, USA, {\tt rboppana@mit.edu}}
\and
Harrie Hendriks\thanks{Department of Mathematics, Radboud University Nijmegen, 
The Netherlands, {\tt H.Hendriks@science.ru.nl}}
\and
Martien C.A. van Zuijlen\thanks{Department of Mathematics, Radboud University Nijmegen, 
The Netherlands, {\tt M.vanZuijlen@science.ru.nl}}
}

\date{\small April 3, 2020}

\begin{document}

\maketitle

\begin{abstract}
Let $v_1$, $v_2$, $\ldots\,$, $v_n$ be real numbers whose squares add up to~$1$.
Consider the $2^n$ signed sums of the form $S = \sum \pm v_i$.
Boppana and Holzman (2017) proved that at least $\frac{13}{32}$ of these sums satisfy $\lvert S \rvert \le 1$.
Here we improve their bound to~$0.427685$.
\end{abstract}

\section{Introduction}

Let $v_1$, $v_2$, \dots, $v_n$ be real numbers such that the sum of their squares is at most~$1$.
Consider the $2^n$ signed sums of the form $S = \pm v_1 \pm v_2 \pm \dots \pm v_n$.
In 1986,
B.~Tomaszewski (see Guy~\cite{Guy}) asked the following question:
is it always true that at least $\frac{1}{2}$ of these sums
satisfy $\abs{S} \le 1$?

Boppana and Holzman~\cite{BH} proved that at least $\frac{13}{32} = 0.40625$ of the sums satisfy $\abs{S} \le 1$.
Actually, they proved a slightly better bound of~$0.406259$.  
See their paper for a discussion of earlier work on Tomaszewski's problem. 

In this note, 
we will improve the bound to~$0.427685$.
We will sharpen the Boppana-Holzman argument by using a Gaussian bound due to Bentkus and Dzindzalieta~\cite{BD}.  

We will use the language of probability.
Let $\Pr[A]$ be the probability of an event~$A$.
A \emph{random sign} is a random variable whose probability distribution is the uniform distribution on the set~$\{ -1, +1 \}$.
With this language, we can 
state our main result.

\begin{main}
Let $v_1$, $v_2$, \dots, $v_n$ be real numbers such that $\sum_{i=1}^n v_i^2 \le 1$.
Let $a_1$, $a_2$, \dots, $a_n$ be independent random signs.
Let $S$ be $\sum_{i = 1}^n a_i v_i$.
Then
$\Pr[ \abs{S} \le 1 ] > 0.427685.$
\end{main}

\section{Proof of the improved bound}

In this section, we will prove the bound of $0.427685$.
We will follow the approach of Boppana and Holzman~\cite{BH},
replacing their fourth-moment method with a Gaussian bound.  

Let $Q$ be the tail function of the standard normal (Gaussian) distribution:
\[
  Q(x) = \frac{1}{\sqrt{2 \pi}} \int_x^\infty e^{-t^2 / 2} \, dt.
\]
Note that $Q$ is a decreasing, positive function.  

Bentkus and Dzindzalieta~\cite{BD} proved the following Gaussian bound on randomly-signed sums. 
See their paper for a discussion of earlier work on such bounds.

\begin{theorem}[Bentkus and Dzindzalieta] \label{thm:BD}
Let $x$ be a real number.
Let $v_1$, $v_2$, \dots, $v_n$ be real numbers such that $\sum_{i=1}^n v_i^2 \le 1$.
Let $a_1$, $a_2$, \dots, $a_n$ be independent random signs.
Let $S$ be $\sum_{i = 1}^n a_i v_i$.
Then
\[
  \Pr[ S \ge x ] \le \frac{Q(x)}{4 Q(\sqrt{2})} \, .
\]	
\end{theorem}

Given a positive number~$c$, 
define $F(c)$ by
\[
  F(c) = \frac{1}{2} - \frac{Q(1/\sqrt{c}\,)}{4 Q(\sqrt{2})} \, .
\]
Note that $F$ is a decreasing function bounded above by~$\frac{1}{2}$.
A calculation shows that $F(\frac{1}{4}) > 0.427685$.

We will need the following lemma, which quantitatively improves Lemma~3 of Boppana and Holzman~\cite{BH}.

\begin{lemma} \label{lemma:drift}
Let $c$ be a positive number.
Let $x$ be a real number such that $\abs{x} \le 1$.
Let $v_1$, $v_2$, \dots, $v_n$ be real numbers such that
\[
  \sum_{i = 1}^n v_i^2 \le c (1 + \abs{x})^2 .
\]
Let $a_1$, $a_2$, \dots, $a_n$ be independent random signs.
Let $Y$ be $\sum_{i=1}^n a_i v_i$.
Then
\[
  \Pr[ \abs{x + Y} \le 1 ] \ge F(c) .
\]
\end{lemma}

\begin{proof}
By symmetry, we may assume that $x \ge 0$.
Let $w_i$ be $\frac{- v_i}{\sqrt{c} \, (1 + x)}$.
Then $\sum_{i = 1}^n w_i^2 \le 1$.
Let $S$ be $\sum_{i=1}^n a_i w_i$.
Then $Y = - \sqrt{c} \, (1 + x) S$.
Because $Y$ has a symmetric distribution, we have
\[
  \Pr[Y > 1 - x] \le \Pr[Y > 0] \le \frac{1}{2} \, .
\]
By the Bentkus-Dzindzalieta inequality (Theorem~\ref{thm:BD}), we have 
\[
  \Pr[Y < -(1 + x)] 
	  = \Pr\Bigl[S > \frac{1}{\sqrt{c}} \, \Bigr] 
		\le \frac{Q(1/\sqrt{c}\,)}{4 Q(\sqrt{2})} \, .
\]
Therefore
\[
  \Pr[ \abs{x + Y} > 1] 
	  = \Pr[Y > 1 - x] + \Pr[Y < -(1 + x)] 
		\le \frac{1}{2} + \frac{Q(1/\sqrt{c}\,)}{4 Q(\sqrt{2})} \, .
\]
Taking the complement, we obtain
\[
  \Pr[ \abs{x + Y} \le 1 ]
	  = 1 - \Pr[ \abs{x + Y} > 1]
	  \ge \frac{1}{2} - \frac{Q(1/\sqrt{c}\,)}{4 Q(\sqrt{2})}
		= F(c).  \qedhere
\]
\end{proof}

We will also need the following lemma,
which says that $F$ satisfies a certain weighted-average inequality.  

\begin{lemma} \label{lemma:average}
Let $K$ be an integer such that $K \ge 2$.  Then
\[
  \frac{1}{2^{K-1}} F\left( \frac{(K + 1)^2 - K}{(2K + 1)^2} \right)
	  + \Bigl(1 - \frac{1}{2^{K-1}} \Bigr) F\left( \frac{(K + 1)^2 - (K+2)}{(2K + 1)^2} \right)
		  \ge F\Bigl( \frac{1}{4} \Bigr) .
\]
\end{lemma}

\begin{proof}
Let
\[
c_1=\frac{(K + 1)^2 - K}{(2K + 1)^2}=\frac14+\frac34\,\frac1{(2K+1)^2} 
\hbox{; }
c_2=\frac{(K + 1)^2 - (K+2)}{(2K + 1)^2}=\frac14-\frac54\,\frac1{(2K+1)^2}.
\]
Since $c_1\ge c_2$ and $F$ is a decreasing function, we see that for $K\ge2$ we have
\[
  \frac{1}{2^{K-1}} F(c_1)
	  + \left(1 - \frac{1}{2^{K-1}} \right) F(c_2)
		  \ge \frac12 F(c_1)+\frac12 F(c_2).
\]
Therefore it is sufficient to show that the following inequality holds for $0\le\xi\le 1/25$:
\begin{equation}\label{Ineq}
\frac12 F\left(\frac14+\frac34\xi\right)+\frac12 F\left(\frac14-\frac54\xi\right)\ge F\left(\frac14\right).
\end{equation}
Once we show that $F(x)$ is a concave function in the region $0 < x \le 1/4+3/100$,
we conclude that the left hand side of the inequality is also concave in $\xi$
in the region $0\le \xi\le 1/25$ and we need only check
the inequality for $\xi=0$ and for $\xi=1/25$.
We will show that $Q(1/\sqrt x \,)$ is convex in $x$ in the region $0< x \le 1/3$.
Recall that $Q$ satisfies the ordinary differential equation $Q''(x)=-xQ'(x)$
and that $Q'(x)<0$ for all $x$.
Thus, for $x>0$
\begin{align*}
\frac{d^2}{dx^2}Q(x^{-1/2})
&=
Q''(x^{-1/2})\Bigl(-\frac12 x^{-3/2}\Bigr)^2+Q'(x^{-1/2})\Bigl(\frac34x^{-5/2}\Bigr)
\\&=
-\frac14Q'(x^{-1/2})x^{-7/2}(1-3x),
\end{align*}
which is positive if $1-3x > 0$.
It follows that $Q(x^{-1/2})$ is convex in the region $0<x\le1/3$.
Therefore $F(x)$ is concave in the region $0 < x \le 1/3$.
Inequality (\ref{Ineq}) holds trivially for $\xi=0$, and one can check by calculation that
it also holds for $\xi=1/25$ (and even for $\xi=1/9$).
\end{proof}

Finally, we will use these two lemmas to prove our main theorem.

\begin{proof}[Proof of Main Theorem]
We will follow the proof of Theorem~4 of Boppana and Holzman~\cite{BH} nearly line for line. 
Their proof uses a different function~$F$.
Closely examining their proof,
we see that they use four properties of~$F$:
it is bounded above by~$\frac{1}{2}$,
satisfies their Lemma~3 (our Lemma~\ref{lemma:drift}),
is a nonincreasing function (on the set of positive numbers),
and satisfies the weighted-average inequality of Lemma~\ref{lemma:average}.
Our function~$F$ has those same four properties.
Hence we reach the same conclusion: $\Pr[ \abs{S} \le 1 ] \ge F( \frac{1}{4})$.
A calculation shows that $F(\frac{1}{4}) > 0.427685$.
\end{proof}

\section*{Acknowledgment}

The first author would like to thank Ron Holzman for fruitful discussions.  
This paper is the result of two independent discoveries of the same improved bound:
one by the first author and one by the second and third authors.


\begin{thebibliography}{9}

\bibitem{BD}
V. K. Bentkus and D. Dzindzalieta.
\newblock A tight {G}aussian bound for weighted sums of {R}ademacher random variables.
\newblock \emph{Bernoulli}, 21(2):1231–-1237, 2015.

\bibitem{BH}
R. B. Boppana and R. Holzman.
\newblock {T}omaszewski's problem on randomly signed sums: breaking the 3/8 barrier.
\newblock \emph{Electronic Journal of Combinatorics}, 24(3):P3.40, 2017.

\bibitem{Guy}
R. K. Guy.
\newblock Any answers anent these analytical enigmas?
\newblock \emph{American Mathematical Monthly}, 93(4):279--281, 1986.

\end{thebibliography}
\end{document}